\documentclass[12pt]{article}
\usepackage[utf8]{inputenc}
\usepackage{amsmath}
\usepackage{amsthm}
\usepackage{amsfonts}
\usepackage{amssymb}
\usepackage[margin=20mm]{geometry}
\usepackage{color}
\usepackage{hyperref}

\def \vol {\operatorname{vol}}

\def\titlerunning#1{\gdef\titrun{#1}}
\makeatletter
\def\author#1{\gdef\autrun{\def\and{\unskip, }#1}\gdef\@author{#1}}

\makeatother
\def\email#1{e-mail: #1}

\def\keywords#1{\par\medskip
\noindent\textbf{Keywords.} #1}

\newcommand{\R}{\mathbb{R}}

\newtheorem{teo}{Theorem}[section]
\newtheorem{lem}[teo]{Lemma}
\newtheorem{rmk}[teo]{Remark}
\newtheorem{defin}[teo]{Definition}

\begin{document}

\baselineskip=17pt
\titlerunning{$L_p$ functional Busemann-Petty centroid inequality}
\title{$L_p$ functional Busemann-Petty centroid inequality}


\author{J. E. Haddad\thanks{Departamento de Matem\'atica, ICEx,  Universidade Federal de Minas Gerais, 30123-970, Belo Horizonte, Brasil; \email{jhaddad@mat.ufmg.br}}, \quad C. H. Jim\'enez\thanks{Departamento de Matem\'atica, Pontif\'icia Universidade Cat\'olica do Rio de Janeiro, 22451-900   Rio de Janeiro, Brasil [\textbf{corresponding author}]; \email{hugojimenez@mat.puc-rio.br}}, \quad L. A. Silva\thanks{Pontif\'icia Universidade Cat\'olica do Rio de Janeiro, Departamento de Matem\'atica, 22451-900, Rio de Janeiro, Brasil. Professora do Ensino B\'asico, T\'ecnico e Tecnol\'ogico no IFMG - Campus Bambu\'i; \email{leticia.alves@ifmg.edu.br}}}

\date{}

\maketitle

%
%

\vspace{0.5cm}
\abstract{
If $K\subset\mathbb{R}^n$ is a convex body and $\Gamma_pK$ is the $p$-centroid body of $K$, the $L_p$ Busemann-Petty centroid inequality states that $\vol(\Gamma_pK) \geq \vol(K)$, with equality if and only if $K$ is an ellipsoid centered at the origin. In this work, we prove inequalities for a type of functional $r$-mixed volume for $1 \leq r < n$, and establish as a consequence, a functional version of the $L_p$ Busemann-Petty centroid inequality. \keywords{Convex body, Moment body, Busemann-Petty centroid} }

\section{Introduction}

The study of affine isoperimetric inequalities on one side and affine Sobolev inequalities for functions on $\R^n$ on the other is connected to a great extent. The equivalence of the classical isoperimetric inequality and the classical $L_1$ Sobolev inequality has been known for quite some time (see for example\cite{aubin1976problemes,talenti1976best,federer1960normal,yuburago1988,osserman1978isoperimetric,maz1960classes,federer1969}). Following this path Zhang in \cite{Zaffsob} established the equivalence of an affine $L_1$ Sobolev inequality with the Petty Projection inequality for convex bodies. Some time after, along with Lutwak and Yang continued in this direction obtaining $L_p$ versions of the mentioned equivalence. These authors developed around the same time a rich theory of geometrical inequalities for centroid bodies and established $L_p$ extensions of many other fundamental parameters in Convex Geometry, such as mixed volume and surface area. 
 
On top of the strong connections mentioned above, other geometrical inequalities of isoperimetric flavour like the Busemann-Petty centroid inequality or Blaschke-Santal\'o, among others, have been fundamental in the study of several inequalities of Sobolev type, like $L_p$ log-Sobolev, Gagliardo-Nirenberg, Sobolev trace or weighted Sobolev inequalities (e.g \cite{HJMaffsobBP,de2018sharp,haddad2017sharp,haddad2018asymmetric,haberl2009asymmetric,HSglpaffiso}). It is important to notice that in many of the works mentioned above, where the Busemann-Petty centroid inequality was used to recover some known results for Sobolev type inequalities, this inequality provided a more direct approach. This approach often went around the use (in their original proofs) of other well known tools in the area of convex geometric analysis like the Minkowski problem or the theory of mixed or dual mixed volumes. 

In this work we continue with this line of research. We obtain a family of inequalities for functions on $\R^n$, inequalities of Sobolev type, and that in particular recover the $L_p$ Busemann-Petty centroid inequality for convex bodies in $\R^n$. Our main inequality is presented in the form of a functional mixed volume inequality.
\begin{teo}
	\label{mainthm}
Let $f$ be a $C^1$ function and $g$ a continuous non-negative function, both with compact support in $\mathbb{R}^n$, then for $1 \leq r < n$, $q = \frac{nr}{n-r}$ and $\lambda \in \left(\frac{n}{n+p}, 1\right)\cup (1, \infty)$,
	\begin{equation}
		\label{mainthm_ineq}
		\int_{\mathbb{R}^n}\left(\int_{\mathbb{R}^n}g(y)|\langle \nabla f(x), y \rangle|^pdy\right)^{r/p}dx \geq C_{n,p, \lambda} ||g||^\frac{[(n+p)(\lambda - 1) + p]r}{np(\lambda - 1)}_1||g||^{-\frac{\lambda r}{(\lambda-1)n}}_{\lambda}||f||^r_{q}.
	\end{equation}
	The sharp constant $C_{n,p, \lambda}$ is computed in Section $3$ and equality is attained if and only if $f$ and $g$ have the following forms
		\[g(x) = aG_{p, \lambda}(||Ax||_2)\]
		\[f(x) = bF_r(||A x||_2)\]
	for positive constants $a, b$, $A \in GL_n(\mathbb{R})$, $G_{p,\lambda}:\R_+ \to \R$ defined by
	\begin{equation*}
	G_{p,\lambda}(t) = \left\{
	\begin{array}{cc}
	(1+t^p)^{\frac 1{\lambda-1}} & \hbox{ if } \lambda \in \left(\frac{n}{n+p}, 1\right) \\
	(1-t^p)_+^{\frac 1{\lambda-1}} & \hbox{ if } \lambda > 1,
	\end{array}
	\right.
	\end{equation*}
	and \[F_r(t) = (1+t^{\frac r{r-1}} )^{1-\frac rn}.\] 
\end{teo}

\section{Some notations and tools from Convex Geometry}

In order to show the intrinsic geometric nature of inequality \eqref{mainthm_ineq}, and in particular, its relation to the $L_p$ Busemann-Petty centroid inequality, let us first recall some basic definitions. A convex body is a convex set $K \subset \mathbb{R}^n$ which is compact and has non-empty interior. For a convex body $K$, its support function $h_K$, which uniquely characterizes it, is defined as
\begin{equation*}
h_K(x) = \max\{\langle x, y\rangle: y \in K \}.
\end{equation*}
If $K$ contains the origin in the interior, then we also have the gauge $\|\cdot\|_K$ and radial $r_K(\cdot)$ functions of $K$ defined respectively as
\[
\|y\|_K:=\inf\{\lambda>0 :\  y\in \lambda K\}\, ,\quad y\in\R^n\setminus\{0\}\, ,
\]

\[
r_K(y):=\max\{\lambda>0 :\ \lambda y\in K\}\, ,\quad y\in\R^n\setminus\{0\}\, .
\]
Clearly, $\|y\|_K=\frac{1}{r_K(y)}$.

For a convex body $K \subset \mathbb{R}^n$ and $p \geq 1$, its $L_p$-moment and $L_p$-centroid bodies, denoted by $M_pK$ and $\Gamma_p K$, are defined by their support functions
\begin{equation}
	\label{def_GpK}
h_{M_pK}(x)^p = \int_{K}|\langle x, y \rangle|^pdy,\quad \mbox{ and }\quad h_{\Gamma_p K}(x)^p =\frac{1}{\vol(K)c_{n,p}} \int_{K}|\langle x, y \rangle|^pdy,
\end{equation}
respectively, where $c_{n,p} = \frac{\omega_{n+p}}{\omega_{2}\omega_{n}\omega_{p-1}}$ and $\omega_m$ is the $m$-dimensional volume of the unit ball $B$ of $\mathbb{R}^m$.
The $L_p$ Busemann-Petty centroid inequality states that    
\begin{equation}\label{bpm}
\vol(\Gamma_pK) \geq \vol(K)\quad \mbox{ or }\quad \vol(M_pK) \geq c^{n/p}_{n,p}\vol(K)^{\frac{n+p}{p}},
\end{equation}
in terms of the moment body $M_pK$.
Equality holds in \eqref{bpm} if and only if $K$ is a $0$-symmetric ellipsoid.

Centroid bodies for $p=1$ can be found for the first time in a work of Blaschke \cite{blaschke1917affine} whereas the respective Busemann-Petty centroid inequality for $p=1$ is due to Petty \cite{petty1961centroid}. The $L_p$ version of centroid bodies above was introduced by Lutwak and Zhang \cite{lutwak1997blaschke}, while (\ref{bpm}) was obtained by Lutwak, Yang and Zhang in \cite{LYZlpaffiso}. For the history of the Busemann-Petty centroid inequality and a comprehensive introduction on centroid and moment bodies we refer to Chapter 10 in \cite{schneider2014convex}. 

The theory of mixed volumes, first developed by Minkowski \cite{MINK1,minkowski1911theorie}, is one of the pillars of the Brunn-Minkowski theory, it provides us with a unified approach to the study of several of the most important parameters in Convex Geometry, such as volume, mean width, surface area, among others. At the same time, it has been fundamental in many other problems ranging from characterization of special families of convex bodies to establish new isoperimetric inequalities, we refer to \cite{schneider2014convex,BONFEN} for a comprehensive introduction to the theory of mixed volumes. There are several extensions of the concept of mixed volume, in this work we will focus mainly in the dual mixed volume and the $L_p$ extension of the mixed volume, concepts belonging to the dual and $L_p$ Brunn-Minkowski theory respectively. Regarding the latter we have the following $L_p$ extension of mixed volume, for some background on this we refer to \cite{LEBMFtheory} and to \cite{LYZoptsobnorms} and the references therein.  

For $r \geq 1$, the $L_r$-mixed volume $V_r(K,L)$ of convex bodies $K$ and $L$ is defined by
\begin{equation*}
V_r(K,  L) = \frac{r}{n}\lim_{\varepsilon \rightarrow 0}\frac{\vol(K +_{r}\varepsilon\cdot_{r}L) - \vol(K)}{\varepsilon},
\end{equation*}
where $K +_{r}\varepsilon \cdot_rL$ is the convex body defined by:
\begin{equation*}
h_{K +_{r}\varepsilon \cdot_r L}(x)^r = h_K(x)^r + \varepsilon h_L(x)^r, \quad \forall x \in \mathbb{R}^n.
\end{equation*}

One of the main aspects of the mixed volume is that it has an integral representation. As in the classical case for the $L_r$ version it is known (see \cite{LEBMFtheory}) that there exists a unique finite positive Borel measure $S_r(K, .)$ on $\mathbb{S}^{n-1}$ such that
\begin{equation}\label{defLpmxvol}
V_r(K, L) = \frac{1}{n}\int_{\mathbb{S}^{n-1}}h_L(u)^rdS_r(K, u),
\end{equation}
for each convex body $L$.

If $1 \leq r < \infty$ and $K, L$ are convex bodies in $\mathbb{R}^n$ containing the origin as interior point, we can find also in \cite{LEBMFtheory} that
\begin{equation}\label{ebvm}
V_r(K, L) \geq \vol(K)^{\frac{n-r}{n}}\vol(L)^{\frac{r}{n}},
\end{equation}	 
with equality if and only if $K$ and $L$ are dilates of each other.  
Combining inequalities \eqref{ebvm}  and \eqref{bpm}, we obtain:
\begin{equation} \label{pp}
V_r(L, M_pK) \geq c^{r/p}_{n,p}\vol(L)^{\frac{n-r}{n}}\vol(K)^{\frac{(n+p)r}{np}}.
\end{equation}
Taking $L = M_pK$ in \eqref{pp}, we recover \eqref{bpm}, hence \eqref{pp} is an equivalent formulation for the $L_p$ Busemann-Petty centroid inequality. This and similar geometric inequalities for mixed volumes involving centroid and projection bodies were already considered in \cite{LutwakMixedPrjIneq}. The main result, Theorem \ref{mainthm} is a functional version of inequality \eqref{pp}, replacing the sets $L,K$ by functions $f,g$.

In order to establish a functional version of \eqref{pp} and considering the integral representation of the geometric $L_r$ mixed volume (\ref{defLpmxvol}), let us recall the following result obtained by Lutwak, Yang and Zhang, where they introduced the concept of surface area measure of a Sobolev function. 

The $L_r$ surface area measure of a function $f: \mathbb{R}^n \rightarrow \mathbb{R}$ with $L_r$ weak derivative is given by:

\begin{lem}[Lemma 4.1 of \cite{LYZoptsobnorms}]\label{lfx}
	Given $1 \leq r < \infty$ and a function $f: \mathbb{R}^n \rightarrow \mathbb{R}$ with $L_r$ weak derivative, there exists a unique finite Borel measure $S_r(f,.)$ on $\mathbb{S}^{n-1}$ such that
	\begin{equation}
		\label{def_Srf}
		\int_{\mathbb{R}^n}\phi(-\nabla f(x))^r dx = \int_{\mathbb{S}^{n-1}}\phi(u)^rdS_r(f,u),
	\end{equation}
for every non-negative continuous function $\phi : \mathbb{R}^n \rightarrow \mathbb{R}$ homogeneous of degree $1$. If $f$ is not equal to a constant function almost everywhere, then the support of $S_r(f,.)$ cannot be contained in any $n-1$ dimensional linear subspace.
\end{lem} 
Conversely, for a convex body $L$ the function $f_L(x) = F(\|x\|_L)$ satisfies $S_r(f,.) = S_r(L,.)$ if $F$ is any function $F:\R_+ \to \R_+$ satisfying 
\[\int_0^\infty t^{n-1}F'(t)^r dt = 1\] (see \cite{LYZoptsobnorms}). 
By the Sobolev inequality we have
\[\int_{\R^n} f_L(x)^{\frac{nr}{n-r}} dx \leq c_s^{\frac{nr}{n-r}} (n \omega_n)^{\frac{n}{n-r}} \frac{\vol(L)}{\omega_n}\]
where $c_s$ is the sharp constant in the Sobolev inequality on $\R^n$, and there is equality when $F(t) = a F_r(t)$ with $a,b > 0$, where
\[F_r(t) = (1+t^{\frac r{r-1}} )^{1-\frac rn}.\]
The function $F(||x||_2)$ is an extremal function of the euclidean $L_r$ Sobolev inequality on $\R^n$.

In view of identity (\ref{def_Srf}),
for any $f$ and $L$ such that $S_r(f,.) = S_r(L,.)$, we have
\[V_r(L, K) = \frac 1n \int_{\mathbb{R}^n} h_K(-\nabla f(x))^r dx.\]
This motivates the following definition.
\begin{defin}
	Given $1 \leq r < \infty$ and a function $f: \mathbb{R}^n \rightarrow \mathbb{R}$ with $L_r$ weak derivative, we define
	\begin{equation*}
		V_r(f, K) = \frac 1n \int_{\mathbb{R}^n}h_K(-\nabla f(x))^r dx
	\end{equation*}
\end{defin}

The $L_p$ Sobolev inequality for general norms was proved in \cite{Cmasstransp} and \cite{alvino1997convex} and can be stated as a mixed volume inequality for functions as follows:
\begin{teo}\label{t1vmv} If $f$ is a $C^1$ function with compact support in $\mathbb{R}^n$ and $K$ is an origin-symmetric convex body, then for $1 < r < n$ and $q = \frac{nr}{n-r}$
	\begin{equation}
		\label{t1vmv_ineq}
		V_r(f, K) \geq c_1^r \|f\|^r_q \vol(K)^{\frac{r}{n}},
	\end{equation}
where $c_1$ is the optimal constant and equality holds in \eqref{t1vmv_ineq} if and only if $f(x) = a F_r(b \|x\|_K)$ for some $a,b>0$.
	Taking $f(x) = F_r(\|x\|_L)$ we recover inequality \eqref{ebvm}.
\end{teo}
	Theorem \ref{t1vmv} was originally proved using an innovative approach based on optimal transportation of mass in \cite{Cmasstransp} and in \cite{alvino1997convex} using Convex Symmetrization.
	
	In Section \ref{sec:MainRes} we give an alternative, simpler and elementary proof of this inequality using the tools developed in \cite{LYZlpaffsob}. Some of the tools we are using here, specially those contained in \cite{LYZoptsobnorms}, have been used in the study of Sobolev type inequalities. Their approach is often based on a functional extension of the so-called $\operatorname{LYZ}$ body and other known geometric inequalities for projection and polar projection bodies (see Subsection 10.15 in \cite{schneider2014convex} and references therein for more on this). 

Let us go back to the definition of the moment body \eqref{def_GpK}, it has been noticed that $h_{M_pK}$ is a convex function regardless of the set $K$ (see e.g. Chapter 5 in \cite{brazitikos2014geometry}).
This observation allows us to make the following definition:
\begin{defin} If $g$ is a non-negative measurable function with compact support, we define the convex body $M_pg$ by	
	\begin{equation*}
		h_{M_pg}(\xi)^p = \int_{\mathbb{R}^n}g(x)|\langle x, \xi \rangle|^pdx.
	\end{equation*}
\end{defin}
The left-hand side of \eqref{mainthm_ineq} has then a geometric meaning:
	\[
	V_r(f, M_pg) = \frac{1}{n} \int_{\mathbb{R}^n}\left(\int_{\mathbb{R}^n}g(y)|\langle \nabla f(x), y \rangle|^pdy\right)^{r/p}dx.
	\]

If $K$ is a convex body and $g(x) = G(\|x\|_K)$ for any non-negative continous function $G:\R_+ \to \R$ with compact support, it is not hard to verify using polar coordinates that
\begin{equation*}\label{mpgmpk}
	M_p g = \left((n+p) \int_0^\infty t^{n+p-1} G(t) dt \right)^{1/p} M_p K.
\end{equation*}

Our main result (Theorem \ref{mainthm}) is a consequence of Theorem \ref{t1vmv}, and Theorem \ref{taux} below:

\begin{teo}
	\label{taux} 
	If $g$ is a non-negative function with compact support in $\mathbb{R}^n$, then, for each $\lambda \in \left(\frac{n}{n+p}, 1\right) \cup (1, \infty)$, we have that
	\begin{equation}
		\label{taux_ineq}
		\vol(M_pg)^{\frac{p}{n}} \geq c_{n,p}a_{n, p, \lambda}||g||^\frac{(n+p)(\lambda - 1) + p}{(\lambda - 1)n}_1||g||^{-\frac{\lambda p}{(\lambda-1)n}}_{\lambda},
	\end{equation}
where $a_{n, p, \lambda}$ is given by the Lemma \eqref{lvnp}. 

	Let $G_{p,\lambda}:\R_+ \to \R$ be defined by
	\begin{equation*}
		G_{p,\lambda}(t) = \left\{
			\begin{array}{cc}
				(1+t^p)^{\frac 1{\lambda-1}} & \hbox{ if } \lambda < 1\\
				(1-t^p)_+^{\frac 1{\lambda-1}} & \hbox{ if } \lambda > 1,
			\end{array}
			\right.
	\end{equation*}
	then taking $g(x) = G_{p,\lambda}(\|x\|_K)$ in \eqref{taux_ineq} we recover \eqref{bpm}.

	Equality holds in \eqref{taux_ineq} if and only if $g(x) = a G_{p, \lambda}(|A.x|_2)$ for any $a > 0$ and $A \in \operatorname{Gl}_n(\R^n)$.
\end{teo}

	Even though Theorem \ref{taux} contains the geometric core of the main Theorem \ref{mainthm}, the term $\vol(M_pg)$ cannot be expressed in terms of $g$ in an elementary way, as $V_r(f, M_pg)$ does. This is the reason why we need to combine it with Theorem \ref{t1vmv} to obtain a functional inequality.
	
	Let us note that Theorem \ref{mainthm} cannot be regarded as a functional mixed volume inequality in full generality since it can only be applied to a function $f$ and the centroid/moment body of another function $g$. We refer the interested reader to review the works of Milman and Rotem \cite{milman2013mixed,milman2013alpha} where they have defined a functional extension of mixed volumes and have extended some of their main properties to a functional setting. 
	
	We should finally also mention other related extension of the Busemann-Petty centroid inequality obtained by Paouris and Pivovarov in \cite{paouris2017randomized}  where the authors obtained randomized versions of this and other important isoperimetric inequalities. 

	The rest of the paper is organized as follows: In Section \ref{Sec:Prelim} we shall prove some preliminary results, including an extension of the $L_p$ Busemann-Petty centroid inequality, to compact domains. Then in Section \ref{sec:MainRes} we prove Theorems \ref{t1vmv} and \ref{taux}.

	We hope this work shed some more light into the deep connection between isoperimetric and functional inequalities.

\section{Preliminary results}\label{Sec:Prelim}

In order to prove our main result, Theorem \ref{mainthm}, we consider two cases: $r=1$ and $1 < r < n$. For $r=1$, inequality \eqref{ebvm} holds for more general sets. As in \cite{Zaffsob}, a compact domain is the closure of a bounded open set.

\begin{lem}[Lemma 3.2 of \cite{Zaffsob}]\label{bmvm} If $M$ is a compact domain with piecewise $C^1$ boundary and $K$ a convex body in $\mathbb{R}^n$, then,
	\begin{equation*}
	V(M, K)^n \geq \vol(M)^{n-1}\vol(K),
	\end{equation*}
with equality if and only if M and K are homothetic.
\end{lem}

In the same spirit, the next lemma shows that the $L_p$-Busemann-Petty Centroid inequality remains valid for a compact domain:

\begin{lem}\label{bpdc}
 If $M$ is a compact domain, then 
 \begin{equation}
	 \label{bpdc_ineq}
 	\vol(\Gamma_pM) \geq \vol(M).
 \end{equation}
Equality holds in \eqref{bpdc_ineq} if and only if $M$ is a $0$-symmetric ellipsoid.
\end{lem}

\begin{proof} For a compact domain $M$ and $\xi \in \mathbb{S}^{n-1}$, we define the set
\begin{equation*}
L_{\xi} = \{t \in [0, \infty): t\xi \in M \}.
\end{equation*}

Consider $\delta (t) = \frac{t^n}{n}$, for $t \geq 0$, and the star set $SM$ defined by its radial function 
\begin{equation*}
\rho_{SM}(\xi) = \delta^{-1}(\mu(\delta(L_{\xi}))),
\end{equation*}
where $\mu$ denotes the one dimensional Lebesgue measure of $L_{\xi}$. It is easy to see that $\vol(SM) = \vol(M)$.
Also, let $s = \delta (t) = \frac{t^n}{n}$, then $ds = t^{n-1}dt$. For $x \in \mathbb{R}^n$, we have:
\begin{align*}
\int_{M}|\langle x, y \rangle|^p dy & = \int_{\mathbb{S}^{n-1}}\int_{L_{\xi}}|\langle x, t\xi \rangle|^pt^{n-1}dtd\xi \\
& = \int_{\mathbb{S}^{n-1}}\int_{L_{\xi}}|\langle x, \xi \rangle|^pt^pt^{n-1}dtd\xi \\
& = \int_{\mathbb{S}^{n-1}}\int_{\delta(L_{\xi})}|\langle x, \xi \rangle|^p(ns)^{\frac{p}{n}}dsd\xi \\
& = n^{{\frac{p}{n}}}\int_{\mathbb{S}^{n-1}}|\langle x, \xi \rangle|^p\int_{\delta(L_{\xi})}s^{\frac{p}{n}}dsd\xi \\
\end{align*}

On the other hand, we have
\begin{align*}
\int_{SM}|\langle x, y \rangle|^p dy & = \int_{\mathbb{S}^{n-1}}\int_{0}^{\rho_{SM}(\xi)}|\langle x, t\xi \rangle|^pt^{n-1}dtd\xi \\
& = \int_{\mathbb{S}^{n-1}}\int_{0}^{\rho_{SM}(\xi)}|\langle x, \xi \rangle|^pt^pt^{n-1}dtd\xi \\
& = \int_{\mathbb{S}^{n-1}}\int_{0}^{\delta({\rho_{SM}(\xi)})}|\langle x, \xi \rangle|^p(ns)^{\frac{p}{n}}dsd\xi \\
& = n^{\frac{p}{n}} \int_{\mathbb{S}^{n-1}}|\langle x, \xi \rangle|^p\int_{0}^{\mu(\delta(L_{\xi}))}s^{\frac{p}{n}}dsd\xi.
\end{align*}

	By the Bathtub principle (see Theorem 1.14, pag. 28 of \cite{lieb2001analysis}) we have

\[\int_{\delta(L_{\xi})}s^{\frac{p}{n}}ds \geq \int_{0}^{\mu(\delta(L_{\xi}))}s^{\frac{p}{n}}ds\]
therefore, 
	\begin{equation}
		\label{radialsim_ineq}
		\int_{M}|\langle x, y \rangle|^p dy \geq \int_{SM}|\langle x, y \rangle|^p dy.
	\end{equation}
Since $\vol(SM) = \vol(M)$, we obtain $h_{\Gamma_pM}(x)^p \geq h_{\Gamma_pSM}(x)^p$, whence $\Gamma_pM \supset \Gamma_pSM$ and  $\vol(\Gamma_pM) \geq \vol(\Gamma_pSM)$. We conclude,
\begin{equation*}\nonumber
	\vol(\Gamma_pM) \geq \vol(\Gamma_pSM) \geq \vol(SM) = \vol(M).
\end{equation*}

	If $M$ is a compact domain attaining equality in \eqref{bpdc_ineq}, then equality in \eqref{radialsim_ineq} implies $\mu(\delta(L_{\xi})) = {\delta(L_{\xi})}$ for a.e $\xi$, meaning that $M$ is a star body. We conclude the proof recalling the equality case of \eqref{bpm}.
\end{proof}

Let $f$ be a $C^1$ function with compact support in $\mathbb{R}^n$. For $t>0$, consider the level sets of $f$ in $\mathbb{R}^n$:
\begin{equation*}
N_{f,t} = \{x \in \mathbb{R}^n: |f(x)| \geq t\}
\end{equation*}
and
\begin{equation*}
	S_{f,t} = \{x \in \mathbb{R}^n: |f(x)| = t \}.
\end{equation*}

Since $f$ is of class $C^1$, by Sard's Theorem, $S_{f,t}$ is a $C^1$ submanifold which has non-zero normal vector $\nabla f$, for almost all $t$. Denote by $dS_t$ the surface area element of $S_{f,t}$. Then the co-area formula relates the area elements $dx = |\nabla f|^{-1}dS_t dt$.

We present a lemma, whose proof is inside of the proof of Theorem $4.1$ of \cite{Zaffsob}. It will be useful to prove Theorem \ref{mainthm} for the case $r=1$.

\begin{lem}\label{lnf} If $f$ is a $C^1$ function with compact support in $\mathbb{R}^n$, then:
	\begin{equation*}
	\int_{0}^{\infty}\vol(N_{f,t})^{\frac{n-1}{n}}dt \geq ||f||_{\frac{n}{n-1}}.
	\end{equation*}
\end{lem}
We observe that the proof of Lemma \ref{lnf} carries over replacing $\frac{n-1}{n}$ by any $\eta \in (0,1)$, but not for $\eta > 1$. We prove an analogous result for $\eta = \frac{n+p}{p} > 1$.

\begin{lem}\label{lvnp} If $g$ is a $C^1$ function with compact support in $\mathbb{R}^n$ and $\lambda \in \left(\frac{n}{n+p}, 1\right)\cup (1, \infty)$
	\begin{equation*} \nonumber
	\int_{0}^{\infty} \vol(N_{g,t})^{\frac{n+p}{n}}dt
	\geq a_{n, p,\lambda}||g||^\frac{(n+p)(\lambda - 1) + p}{(\lambda - 1)n}_1||g||^{-\frac{\lambda p}{(\lambda-1)n}}_{\lambda},
	\end{equation*}
	where
	
$$
a_{n, p, \lambda} = \left\{
	\begin{array}{lr}
	A^{-\frac{(n+p)(\lambda - 1) + p}{(\lambda - 1)n}}_{n, p, \lambda} & \text{if } \lambda > 1\\
	B^{\frac{p}{(\lambda -1)n}}_{n, p, \lambda} & \text{if } \lambda \in \left(\frac{n}{n+p}, 1\right)
	\end{array} \right.
$$	
with	

\[
	A_{n, p, \lambda} = ((\lambda -1) n+\lambda  p) \left(\frac{\Gamma \left(\frac{\lambda }{\lambda -1}\right) (\lambda  p)^{\frac{1}{1-\lambda }} ((\lambda -1) (n+p))^{-\frac{n+p}{p}} \Gamma \left(\frac{n}{p}+2\right)}{\Gamma \left(\frac{n}{p}+\frac{1}{\lambda -1}+2\right)}\right)^{\frac{(\lambda -1) p}{(\lambda -1) n+\lambda  p}}
\]
and
\[
	B_{n,p, \lambda} = \lambda \frac{p}{n+p} \left(\lambda -\frac{n}{n+p}\right)^{\frac{(1-\lambda) (n+p)}{p}-1} \left(\frac{(1-\lambda )^{-\frac{n}{p}-2} \Gamma \left(\frac{n}{p}+2\right) \Gamma \left( \frac{\lambda}{1-\lambda}-\frac np \right)}{\Gamma \left(\frac{\lambda -2}{\lambda -1}\right)}\right)^{1-\lambda }
\]

\end{lem}
\begin{proof}
For $\lambda > 1$ and $t>0$, let $p_{\lambda}(t) = (1-t^{\lambda-1})^{\frac{n}{p}}_+$ and $l(t)= \vol(N_{g,t})$. Then $p_{\lambda}\left(\frac{t}{s}\right) = (1-t^{\lambda-1}s^{1-\lambda})^{\frac{n}{p}}_+$ and
\begin{equation}\label{ppn}
p_{\lambda}\left(\frac{t}{s}\right)^{\frac{p}{n}} \geq 1-t^{\lambda-1}s^{1-\lambda}.
\end{equation}
Multiplying \eqref{ppn} by $l(t)$ and integrating, we obtain:
\begin{equation*}
\int_{0}^{\infty}l(t)p_{\lambda}\left(\frac{t}{s}\right)^{\frac{p}{n}}dt \geq \int_{0}^{\infty}l(t)dt - s^{1-\lambda}\int_{0}^{\infty}l(t)t^{\lambda-1}dt,
\end{equation*}	
whence
\begin{equation*}
||g||_1 \leq \int_{0}^{\infty}l(t)p_{\lambda}\left(\frac{t}{s}\right)^{\frac{p}{n}}dt + s^{1-\lambda}\int_{0}^{\infty}l(t)t^{\lambda-1}dt.
\end{equation*}

By Holder, observe that:

\begin{equation*}
\int_{0}^{\infty}l(t)p_{\lambda}\left(\frac{t}{s}\right)^{\frac{p}{n}}dt \leq  \left(\int_{0}^{\infty} l(t)^{\frac{n+p}{n}}dt\right)^{\frac{n}{n+p}}\left(\int_{0}^{\infty}p_{\lambda}\left(\frac{t}{s}\right)^{\frac{n+p}{n}}dt\right)^{\frac{p}{n+p}}.
\end{equation*}

Write $u = t/s$ and $dt = sdu$. Then:
\begin{equation*}
\int_{0}^{\infty}l(t)p_{\lambda}\left(\frac{t}{s}\right)^{\frac{p}{n}}dt \leq  \left(\int_{0}^{\infty} l(t)^{\frac{n+p}{n}}dt\right)^{\frac{n}{n+p}}\left(\int_{0}^{\infty}p_{\lambda}(u)^{\frac{n+p}{n}}du\right)^{\frac{p}{n+p}}s^{\frac{p}{n+p}}.
\end{equation*}

Now, observe that:
\begin{equation*}
\int_{0}^{\infty}l(t)t^{\lambda-1}dt = \int_{0}^{\infty}\vol(N_{g^{\lambda}, t^{\lambda}})t^{\lambda -1}dt.
\end{equation*}

Write $v = t^{\lambda}$, $dv = \lambda t^{\lambda -1}dt$, then $t^{\lambda -1}dt = \frac{1}{\lambda}dv$ and:
\begin{equation*}
\int_{0}^{\infty}l(t)t^{\lambda-1}dt = \frac{1}{\lambda}\int_{0}^{\infty}\vol(N_{g^{\lambda}, t})dt = \frac{1}{\lambda}||g||^{\lambda}_{\lambda}.
\end{equation*}

Hence,
\begin{equation}\label{minimizar}
||g||_1 \leq \frac{1}{\lambda}||g||^{\lambda}_{\lambda}s^{1-\lambda} + \left(\int_{0}^{\infty} l(t)^{\frac{n+p}{n}}dt\right)^{\frac{n}{n+p}}\left(\int_{0}^{\infty}p_{\lambda}(t)^{\frac{n+p}{n}}dt\right)^{\frac{p}{n+p}}s^{\frac{p}{n+p}} = as^{-\alpha} + bs^{\beta},
\end{equation}
where $a = \frac{1}{\lambda}||g||^{\lambda}_{\lambda}$, $b = \left(\int_{0}^{\infty} l(t)^{\frac{n+p}{n}}dt\right)^{\frac{n}{n+p}}\left(\int_{0}^{\infty}p_{\lambda}(t)^{\frac{n+p}{n}}dt\right)^{\frac{p}{n+p}}$, $\alpha = \lambda - 1$ e $\beta = \frac{p}{n+p}$.
\vspace{0.5cm}

Notice that the right-hand side of $(\ref{minimizar})$ has a unique minimum for $s \in (0, \infty)$, then minimizing with respect to $s \in (0, \infty)$, we obtain:
\begin{equation*}
||g||_1 \leq A_{n, p, \lambda}||g||^{\frac{\lambda p}{(n+p)(\lambda - 1) + p}}_{\lambda}\left(\int_{0}^{\infty} l(t)^{\frac{n+p}{n}}dt\right)^{\frac{(\lambda - 1)n}{(n+p)(\lambda - 1) + p}},
\end{equation*}
where $A_{n, p, \lambda}$ is given in the statement of the Lemma.

Hence,
\begin{equation*}
||g||_1 \leq A_{n, p, \lambda}||g||^{\frac{\lambda p}{(n+p)(\lambda - 1) + p}}_{\lambda}\left(\int_{0}^{\infty} \vol(N_{g,t})^{\frac{n+p}{n}}dt\right)^{\frac{(\lambda - 1)n}{(n+p)(\lambda - 1) + p}},
\end{equation*}
that proves the statement of Lemma for the case $\lambda >1$.

For the case $\lambda \in \left(\frac{n}{n+p}, 1\right)$,
we define $q_{\lambda}(t) = (t^{\lambda-1} - 1)_+^{\frac{n}{p}}$. Then, $q_{\lambda}(t)^{\frac{p}{n}} \geq t^{\lambda-1} - 1$ and $q_{\lambda}\left(\frac{t}{s}\right)^{\frac{p}{n}} \geq t^{\lambda-1}s^{1-\lambda} - 1$.

It follows that

\begin{equation*}
\int_{0}^{\infty}l(t)q_{\lambda}\left(\frac{t}{s}\right)^{\frac{p}{n}}dt \geq s^{1 - \lambda}\int_{0}^{\infty}l(t)t^{\lambda - 1}dt - \int_{0}^{\infty}l(t)dt
\end{equation*}

Since $\int_{0}^{\infty}l(t)t^{\lambda - 1}dt = \frac{1}{\lambda}||g||^{\lambda}_{\lambda}$ and $\int_{0}^{\infty}l(t)dt = ||g||_1$, we obtain

\begin{equation*}
\frac{s^{1 - \lambda}}{\lambda}||g||^{\lambda}_{\lambda} \leq ||g||_1 +	\int_{0}^{\infty}l(t)q_{\lambda}\left(\frac{t}{s}\right)^{\frac{p}{n}}dt.
\end{equation*}

By Holder

\begin{align} \nonumber
\int_{0}^{\infty}l(t)q_{\lambda}\left(\frac{t}{s}\right)^{\frac{p}{n}}dt & \leq  \left(\int_{0}^{\infty} l(t)^{\frac{n+p}{n}}dt\right)^{\frac{n}{n+p}}\left(\int_{0}^{\infty}q_{\lambda}\left(\frac{t}{s}\right)^{\frac{n+p}{n}}dt\right)^{\frac{p}{n+p}} \\ \nonumber
& = \left(\int_{0}^{\infty} l(t)^{\frac{n+p}{n}}dt\right)^{\frac{n}{n+p}}\left(\int_{0}^{\infty}q_{\lambda}(u)^{\frac{n+p}{n}}du\right)^{\frac{p}{n+p}}s^{\frac{p}{n+p}}
\end{align}

Hence,
\begin{equation}\label{minimizar2}
\frac{1}{\lambda}||g||^{\lambda}_{\lambda} \leq s^{ \lambda - 1}||g||_1 + \left(\int_{0}^{\infty} l(t)^{\frac{n+p}{n}}dt\right)^{\frac{n}{n+p}}\left(\int_{0}^{\infty}q_{\lambda}(t)^{\frac{n+p}{n}}dt\right)^{\frac{p}{n+p}}s^{\frac{p}{n+p} + \lambda -1}.	
\end{equation}

For $\lambda \in \left(\frac{n}{n+p}, 1\right)$, the right-hand side of $(\ref{minimizar2})$ has a unique minimum $s \in (0, \infty)$, then minimizing with respect to $s \in (0, \infty)$, we obtain:
\begin{equation*}
	||g||^{\lambda}_{\lambda} \leq B_{n,p, \lambda}||g||^{\frac{(n+p)(\lambda - 1) + p}{p}}_1\left(\int_{0}^{\infty} l(t)^{\frac{n+p}{n}}dt\right)^{\frac{(1 -\lambda)n}{p}},
\end{equation*}
where $B_{n, p, \lambda}$ is given in the statement of the Lemma.

Therefore,
\begin{equation*}
\int_{0}^{\infty} \vol(N_{g,t})^{\frac{n+p}{n}}dt \geq B^{-\frac{p}{(1-\lambda)n}}_{c, d, \lambda}||g||^{-\frac{(n+p)(\lambda - 1) + p}{(1-\lambda)n}}_1 ||g||^{\frac{p\lambda}{(1-\lambda)n}}_{\lambda}.
\end{equation*}
\end{proof}

Now, we present other tools for the case $1< r < n$  of our main result, introduced by Lutwak, Yang, and Zhang in \cite{LYZlpaffsob}. Let $H^{1,r}(\mathbb{R}^n)$ denote the usual Sobolev space of real-valued functions of $\mathbb{R}^n$ with $L_r$ partial derivatives. If $f \in H^{1,r}(\mathbb{R}^n) \cap C^{\infty}(\mathbb{R}^n)$ and $Q$ is a compact convex set that contains the origin in its relative interior, then they define
\begin{equation*}
	V_r(f, t, Q) = \frac{1}{n}\int_{S_{f,t}}h_Q(\nu(x))^r|\nabla f(x)|^{r-1}dS_t(x),
\end{equation*}
where $\nu(x) = \frac{\nabla f(x)}{|\nabla f(x)|}$. They prove that for almost every $t > 0$, there exists an origin-symmetric convex body $K_t$ such that, for each origin-symmetric convex body Q
\begin{equation}\label{vfc}
V_r(K_t, Q) = V_r(f, t, Q).
\end{equation}
The next Lemma can be deduced from \cite{LYZlpaffsob}, inequalities $(6.3)$, $(5.3)$, $(5.4)$, and $(5.1)$.
\begin{lem}\label{ivpv} If $r \in (1, n)$, $f \in H^{1,r}(\mathbb{R}^n)$ and $q = \frac{nr}{n-r}$, then	
	\begin{equation*}
	\int_{0}^{\infty}\vol(K_t)^{\frac{n-r}{n}}dt \geq n^{\frac{r-n}{n}}c_2^r||f||^r_q,
	\end{equation*}
where
\begin{equation*}
c_2 = n^{\frac{1}{q}}\left(\frac{n-r}{r-1}\right)^{\frac{r-1}{r}}\left(\frac{\Gamma\left(\frac{n}{r}\right)\Gamma\left(n+1 - \frac{n}{r}\right)}{\Gamma(n)}\right)^{\frac{1}{n}}.
\end{equation*}	
\end{lem}

\section{Proof of the main results}\label{sec:MainRes}
We present separate proofs for the cases $1< r <n$ and $r=1$.
\subsection{Case $1 < r < n$:}
\begin{proof}[Proof of \ref{t1vmv}]

By the co-area formula, \eqref{vfc}, \eqref{ebvm} and Lemma \ref{ivpv}:
\begin{align*} \nonumber
V_r(f, K) 
	& = \frac 1n \int_{\R^n} h_K(-\nabla f(x))^r dx \\
	& = \frac 1n \int_{0}^{\infty}\int_{S_{f,t}} h_{K}(n^{S_{f,t}}_x)^r|\nabla f(x)|^{r-1}dS_{f,t} dt \\ \nonumber
& = \int_{0}^{\infty}V_r(f, t, K)dt \\ \nonumber
& = \int_{0}^{\infty}V_r(K_t, K)dt \\ \nonumber
& \geq \int_{0}^{\infty}\vol(K_t)^{\frac{n-r}{n}}\vol(K)^{\frac{r}{n}}dt \\ \nonumber
& = \int_{0}^{\infty}\vol(K_t)^{\frac{n-r}{n}}dt \vol(K)^{\frac{r}{n}}\\ 
& \geq n^{\frac{r-n}{n}}c_2^r||f||^r_q\vol(K)^{\frac{r}{n}}
\end{align*}
\end{proof}

\begin{proof}[Proof of \ref{taux}]
	We may observe that:
	\begin{align} \nonumber
	h_{M_pg}(\xi)^p 
		&= \int_{\mathbb{R}^n}g(x)|\langle x, \xi \rangle|^pdx \\ \nonumber
		&= \int_{0}^{\infty}\int_{\{g \geq t\}}|\langle x, \xi \rangle|^pdxdt \\ \nonumber
		&= \int_{0}^{\infty}h_{M_pN_{g,t}}(\xi)^pdt.
	\end{align}
	
	In this sense, we regard $M_pg$ as a generalized $p$-sum of sets, where we replace finite $p$-sums by a $p$-integral of sets
	\[M_pg = \int_p M_pN_{g,t} dt\]
	and clearly, for any convex body $K$,
	\[V_p\left(K, \int_p M_pN_{g,t} dt\right) = \int_0^\infty V_p(K, M_pN_{g,t}) dt.\]
	We compute:
	\begin{align*}
		\vol(M_pg)
		&= V_p(M_pg, M_pg)\\
		&= V_p\left(M_pg, \int_p M_pN_{g,t} dt\right)\\
		&= \int_0^\infty V_p(M_pg, M_p N_{g,t}) dt\\
		&\geq \vol(M_pg)^{\frac{n-p}n} \int_0^\infty \vol(M_p N_{g,t})^{p/n} dt\\
	\end{align*}

	Then using Lemmas \ref{bpdc} and \ref{lvnp}, it follows that
	\begin{align}\nonumber
	\vol(M_pg)^{\frac{p}{n}} 
	& \geq \int_{0}^{\infty}\vol(M_pN_{g,t})^{\frac{p}{n}}dt \\ \nonumber
	& \geq c_{n,p}\int_{0}^{\infty}\vol(N_{g,t})^{\frac{n+p}{n}}dt \\ \nonumber
	& \geq  c_{n,p}a_{n,p,\lambda}||g||^\frac{(n+p)(\lambda - 1) + p}{(\lambda - 1)n}_1||g||^{-\frac{\lambda p}{(\lambda-1)n}}_{\lambda},
	\end{align}
where $a_{n,p,\lambda}$ is given by Lemma \eqref{lvnp}.

\end{proof}

\subsection{Proof of Theorem \ref{mainthm}: Case $r=1$}
\begin{proof}
Let $V_1(f, M_pg) = \frac{1}{n} \int_{\mathbb{R}^n}\left(\int_{{R}^n}g(y)|\langle \nabla f(x), y \rangle|^pdy\right)^{1/p}dx$. Then,

\begin{equation*}
	V_1(f, M_pg) = \frac{1}{n} \int_{0}^{\infty}\int_{S_{f,t}}\left(\int_{{R}^n}g(y)\left|\left\langle \frac{\nabla f(x)}{|\nabla f(x)|}, y \right\rangle \right|^pdy\right)^{1/p}dS_t dt.
\end{equation*}

We denote $\eta^{S_t}_{x} = \frac{\nabla f(x)}{|\nabla f(x)|}$, as
\begin{align*}
h_{M_pg}(\eta^{S_t}_x) & =  \left(\int_{{R}^n}g(y)\left|\left\langle \eta^{S_t}_x, y \right\rangle \right|^pdy\right)^{1/p} \\
& =  \left(\int_{0}^{\infty}\int_{N_{g,s}}\left|\left\langle \eta^{S_t}_x, y \right\rangle \right|^pdyds\right)^{1/p}, 
\end{align*}
it follows that:
\begin{align*}
V_1(f, M_pg) & = \frac{1}{n} \int_{0}^{\infty}\int_{S_{f,t}}h_{M_pg}(\eta^{S_t}_x)dS_tdt \\
& = \frac{1}{n}\int_{0}^{\infty}\int_{S_{f,t}} \left(\int_{0}^{\infty}\int_{N_{g,s}}\left|\left\langle \eta^{S_t}_x, y \right\rangle \right|^pdyds\right)^{1/p}dS_tdt 
\end{align*}

Write $h_{M_pN_{g,s}}(\eta^{S_t}_x)^p = \int_{N_{g,s}}\left|\left\langle \eta^{S_t}_x, y \right\rangle \right|^pdy$, then:

\begin{equation*}
V_1(f, M_pg) = \frac{1}{n}\int_{0}^{\infty}\int_{S_{f,t}} \left(\int_{0}^{\infty}h_{M_pN_{g,s}}(\eta^{S_t}_x)^p ds\right)^{1/p}dS_tdt
\end{equation*}

By the co-area formula, the Minkowski integral inequality and Lemmas \ref{bmvm}, \ref{bpdc}, \ref{ivpv} and \ref{lvnp}:
\begin{align*} \nonumber
V_1(f, M_pg) & \geq \frac{1}{n} \int_{0}^{\infty}\left(\int_{0}^{\infty}\left(\int_{S_{f,t}} h_{M_pN_{g,s}}(\eta^{S_t}_x)dS_t\right)^{p}ds\right)^{\frac{1}{p}}dt \\ \nonumber
&\geq \frac{1}{n}\left( \int_{0}^{\infty}\left(\int_{0}^{\infty}\left(\int_{S_{f,t}} h_{M_pN_{g,s}}(\eta^{S_t}_x)dS_t\right)dt\right)^{p}ds\right)^{\frac{1}{p}}
\\ \nonumber
& = \left(\int_{0}^{\infty}\left(\int_{0}^{\infty}V_{1}(N_{f,t}, M_pN_{g,s})dt\right)^pds\right)^{\frac{1}{p}} \\ \nonumber
& \geq  \left(\int_{0}^{\infty}\left(\int_{0}^{\infty}\vol(N_{f,t})^{\frac{n-1}{n}}\vol(M_pN_{g,s})^{\frac{1}{n}}dt\right)^pds\right)^{\frac{1}{p}}\\ \nonumber
& = \left(\int_{0}^{\infty}\left(\int_{0}^{\infty}\vol(N_{f,t})^{\frac{n-1}{n}}dt\right)^p\vol(M_pN_{g,s})^{\frac{p}{n}}ds\right)^{\frac{1}{p}} \\ \nonumber
& = \left(\int_{0}^{\infty}\vol(N_{f,t})^{\frac{n-1}{n}}dt\right)\left(\int_{0}^{\infty}\vol(M_pN_{g,s})^{\frac{p}{n}}ds\right)^{\frac{1}{p}} \\ \nonumber
& \geq c^{\frac{1}{p}}_{n,p}\left(\int_{0}^{\infty}\vol(N_{f,t})^{\frac{n-1}{n}}dt\right)\left(\int_{0}^{\infty}\vol(N_{g,s})^{\frac{n+p}{n}}ds\right)^{\frac{1}{p}} \\
& \geq c^{1/p}_{n,p}||f||_{\frac{n}{n-1}}C^{-\frac{(n+p)(\lambda - 1) + p}{(\lambda - 1)np}}_{n, p, \lambda}||g||^\frac{(n+p)(\lambda - 1) + p}{(\lambda - 1)np}_1||g||^{-\frac{\lambda}{(\lambda-1)n}}_{\lambda}.
\end{align*}
\end{proof}

\begin{rmk}
	Let us point out that a simpler proof of Theorem \ref{mainthm} for the case $r = p$ can be deduced using the $L_p$ Affine Sobolev inequality \cite{LYZlpaffsob} and the equivalence between the $L_p$ Busemann-Petty centroid inequality and the $L_p$ Petty projection inequality (see \cite{LYZlpaffiso}).
	The well known identity for sets
	\[V_p(L,\Gamma_p K) = \frac{\omega_n}{\vol(K)}\tilde V_{-p}(K, \Pi_p^\circ L),\]
	where $\tilde V_{p}(\cdot,\cdot)$ denotes the $L_p$ dual mixed volume and  $\Pi_p^\circ L$ the $L_p$ polar projection body of $L$, can be extended to functions as
	\[V_p(f,M_p g) = \tilde V_{-p}(g, \Pi_p^\circ f)\]
	where we define 
	\[\tilde V_{-p}(g, L) = \int_{\R^n} ||x||_L^p g(x) dx\]
	and
	\[h(\Pi_p^\circ f, \xi)^p = \int_{\R^n} |\langle \nabla f(x), \xi)\rangle|^p dx.\]
%
	Then an application of the dual mixed volume inequality for functions (Lemma 4.1 in \cite{LYZmoment}) and the $L_p$ Affine Sobolev inequality (which corresponds to the $L_p$ Petty Projection inequality for functions), gives the result.
\end{rmk}

\section*{Acknowledgements} The first author was partially supported by Fapemig, Project APQ-01542-18 and CNPQ grant PQ-301203/2017-2. The second and third authors are partially supported by FAPERJ grant JCNE 236508 and CNPQ grant 428076/2018-1. The second author was also partially supported by CNPQ grant PQ 305650/2016-5 and PUC-Rio programa de incentivo a produtividade em pesquisa. The third author acknowledges the support of the IFMG campus Bambui while conducting this work.

\bibliographystyle{plain}

\bibliography{biblioteca}

\end{document}